\documentclass[11pt,a4paper]{amsart}
\usepackage[utf8]{inputenc}
\usepackage[english]{babel}
\usepackage{amssymb,times, centernot}
\usepackage{graphicx}
\usepackage[pdftex,usenames,dvipsnames]{color}
\usepackage{xcolor}
\usepackage[all]{xy}
\usepackage{tikz}
\usepackage[bookmarks=true,hyperindex,pdftex,colorlinks, citecolor=blue]{hyperref}
\usepackage{url}
\usepackage{soul}

\newcommand{\R}{{\mathbb{R}}}

\newcommand{\N}{{\mathbb{N}}}

\newcommand{\eps}{\varepsilon}

\newcommand{\conv}{{\mathrm{conv}}}

\newcommand{\bea}{\begin{eqnarray*}}
\newcommand{\eea}{\end{eqnarray*}}
\newcommand{\beq}{\begin{equation}}
\newcommand{\eeq}{\end{equation}}
\newcommand{\norm}[1]{\left\Vert#1\right\Vert}

\numberwithin{equation}{section}

\theoremstyle{plain}
\newtheorem{theorem}{Theorem}[section]

\newtheorem{prop}[theorem]{Proposition}
\newtheorem{lemma}[theorem]{Lemma}

\newtheorem{probl}[theorem]{Problem}

\theoremstyle{definition}

\newtheorem{remark}[theorem]{Remark}

\newtheorem{definition}[theorem]{Definition}

\renewcommand{\leq}{\leqslant}
\renewcommand{\geq}{\geqslant}
\renewcommand{\le}{\leqslant}
\renewcommand{\ge}{\geqslant}

\begin{document}

\title[Quantitative version of the Bishop-Phelps-Bollob\'{a}s theorem for operators]{Quantitative version of the Bishop-Phelps-Bollob\'{a}s theorem for operators with values in a space with the property $\beta$}
\author[Kadets]{Vladimir Kadets}
\author[Soloviova]{Mariia Soloviova}

\address[Kadets]{Department of Mathematics and Informatics, Kharkiv V.~N.~Karazin National University, pl.~Svobody~4,
61022~Kharkiv, Ukraine
\newline
\href{http://orcid.org/0000-0002-5606-2679}{ORCID: \texttt{0000-0002-5606-2679} } }
\email{vova1kadets@yahoo.com}

\address[Soloviova]{Department of Mathematics and Informatics, Kharkiv V.~N.~Karazin National University, pl.~Svobody~4,
61022~Kharkiv, Ukraine
\newline
\href{http://orcid.org/0000-0002-3777-5286}{ORCID: \texttt{0000-0002-3777-5286} } }

\email{mariiasoloviova93@gmail.com}




\date{2017}
\begin{abstract}

The  Bishop-Phelps-Bollob\'{a}s property for operators deals with simultaneous approximation of an operator $T$ and a vector $x$ at which $T: X\rightarrow Y$ nearly attains its norm by an operator $F$ and a vector $z$, respectively, such that $F$ attains its norm at $z$. We study the possible estimates from above and from below for parameters that measure the rate of approximation in the Bishop-Phelps-Bollob\'{a}s property for operators for the case of $Y$ having the property $\beta$ of Lindenstrauss.
\end{abstract}
\maketitle


\section{Introduction}
In this paper $X$, $Y$ are real Banach spaces,  $L(X,Y)$ is the space of all bounded linear operators $T\colon X \to Y$, $L(X) = L(X,X)$, $X^* = L(X,\R)$, $B_X$ and $S_X$ denote the closed unit ball and the unit sphere of $X$, respectively. 
A functional $x^* \in X^*$ \emph{attains its norm}, if there is $x \in S_X$ with $x^*(x) = \|x^*\|$.  
The \emph{Bishop-Phelps theorem} \cite{Bishop-Phelps} (see also \cite[Chapter 1, p. 3]{Diestel}) says that the set of norm-attaining functionals is always dense in $X^*$. In  \cite{Bollobas} B.~Bollob\'{a}s  remarked that in fact the  Bishop-Phelps construction allows to approximate at the same time a functional and a vector in which it almost attains the norm. Nowadays this very useful fact is called the \emph{ Bishop-Phelps-Bollob\'{a}s theorem}.  Recently, two moduli have been introduced \cite{Modul2014} which measure, for a given Banach space, what is the best possible Bishop-Phelps-Bollob\'{a}s theorem in that space. We will use the following notation:
$$
\Pi(X):=\bigl\{(x,x^*)\in X\times X^*\,:\,
\|x\|=\|x^*\|=x^*(x)=1\bigr\}.
$$
\begin{definition}[Bishop-Phelps-Bollob\'{a}s moduli, \cite{Modul2014}] \label{def1.2bpb-mod}  Let $X$ be a Banach space. The \emph{Bishop-Phelps-Bollob\'{a}s modulus} of $X$ is the function $\Phi_X:(0,2)\longrightarrow \R^+$ such that given $\eps\in (0,2)$, $\Phi_X(\eps)$ is the infimum of those $\delta>0$ satisfying that for every $(x,x^*)\in B_X\times B_{X^*}$ with $ x^*(x)>1-\eps$, there is $(y,y^*)\in \Pi(X)$ with  $\|x-y\|<\delta$ and $\|x^*-y^*\|<\delta$. 
Substituting $(x,x^*)\in S_X\times S_{X^*}$ instead of $(x,x^*)\in B_X\times B_{X^*}$  in the above sentence, we obtain the definition of the \emph{spherical Bishop-Phelps-Bollob\'{a}s modulus}   $\Phi^S_X(\eps)$. 
\end{definition}
Evidently, $\Phi^S_X(\eps)\leq\Phi_X(\eps)$.
There is a common upper bound for $\Phi_X(\cdot)$ (and so for $\Phi_X^S(\cdot)$) for all Banach spaces which is actually sharp. Namely \cite{Modul2014},  for every Banach space $X$ and every $\eps\in (0,2)$ one has
$\Phi_X(\eps)\leq \sqrt{2\eps}$. In other words, this leads to the following improved version of the Bishop-Phelps-Bollob\'{a}s theorem.

\begin{prop}[\mbox{\cite[Corollary~2.4]{Modul2014}}] \label{Prop1.2}
Let $X$ be a Banach space and $0<\eps<2$. Suppose that $x\in B_X$ and $x^*\in B_{X^*}$ satisfy $x^*(x) > 1- \eps$. Then, there exists $(y,y^*)\in \Pi(X)$ such that $\|x-y\|<\sqrt{2\eps}$ and $\|x^*-y^*\|<\sqrt{2\eps}$.
\end{prop}

The sharpness of this version is demonstrated in \cite[Example~2.5]{Modul2014} by just considering $X=\ell_1^{(2)}$, the two-dimensional real $\ell_1$ space. For a uniformly non-square Banach space $X$ one has $\Phi_X(\eps)< \sqrt{2\eps}$ for all $\eps~\in~(0,2)$ (\cite[Theorem~5.9]{Modul2014}, \cite[Theorem~2.3]{Modul2016}).  A quantifcation of this inequality in terms of a parameter that measures the uniform non-squareness of $X$ was given in  \cite[Theorem 3.3]{Modul2015}.

Lindenstrauss in \cite{Lind} examined the extension of the Bishop--Phelps theorem on denseness of the family of norm-attaining scalar-valued functionals on
a Banach space, to vector-valued linear operators.
He introduced the property $\beta$, which is possessed by polyhedral finite-dimensional spaces, and by any subspace of  $\ell_\infty$ that contains $c_0$. 

\begin{definition} \label{def:beta} A Banach space $Y$ is said to have the property $\beta$ if there are two sets $\{y_\alpha: \alpha\in \Lambda\}\subset S_Y$, $\{y^*_\alpha: \alpha\in  \Lambda\}\subset S_Y^*$ and $0\leq \rho < 1$ such that the following conditions hold
\begin{enumerate}
\item $y^*_\alpha(y_\alpha) = 1$,
\item $|y^*_\alpha(y_\gamma)| \leq \rho$ if $\alpha \neq \gamma$,
\item $\norm{y} = \sup\{|y_\alpha^*(y)|: \alpha \in \Lambda\}$, for all $y\in Y$.
\end{enumerate}
\end{definition}
Denote for short by $\beta(Y)\leq\rho$ that a Banach space $Y$ has the property $\beta$ with parameter $\rho\in(0, 1)$. Obviously, if $\rho_1\leq\rho_2<1$ and $\beta(Y)\leq\rho_1$ , then $\beta(Y)\leq\rho_2$. If $Y$ has the property $\beta$ with parameter $\rho = 0$, we will write $\beta(Y)=0$.

Lindenstrauss proved that if a Banach space $Y$ has the property $\beta$, then for any Banach space $X$ the set of norm attaining operators is dense in $L(X,Y)$.  It was proved later by J. Partington \cite{Partin} that every Banach space can be equivalently renormed to have the property $\beta$.

In 2008, Acosta, Aron, Garc\'{i}a and Maestre in \cite{Acosta2008} introduced the following Bishop-Phelps-Bollob\'{a}s property as an extension of the Bishop-Phelps-Bollob\'{a}s theorem to the vector-valued case.

\begin{definition} \label{defBPBproperty} A couple of Banach spaces $(X,Y)$ is said to have \emph{the Bishop-Phelps-Bollob\'{a}s property for operators} if for any 
$\delta>0$
there exists a $\eps(\delta)>0$, such that for every operator $T\in S_{L(X,Y)}$, 
if $x\in S_X$ and
$\|T(x)\|>1-\eps(\delta)$, then there exist $z\in S_X$ and
$F\in S_{L(X,Y)}$ satisfying  $\|F(z)\|=1, \|x-z\|<\delta$  and $\|T-F\|<\delta$.
\end{definition}

In \cite[Theorem 2.2]{Acosta2008} it was proved that if $Y$ has the property $\beta$, then for any Banach space $X$ the pair $(X, Y)$ has the Bishop-Phelps-Bollob\'{a}s property for operators. In this article we introduce an analogue of the Bishop-Phelps-Bollob\'{a}s moduli for the vector-valued case.

\begin{definition}\label{def-bpb-mod-op}
Let $X, Y$ be Banach spaces. The \emph{Bishop-Phelps-Bollob\'{a}s modulus} (\emph{spherical Bishop-Phelps-Bollob\'{a}s modulus}) of a pair $(X, Y)$ is the function $\Phi(X,Y,\cdot):(0,1)\longrightarrow \R^+$ ($\Phi^S(X,Y,\cdot):(0,1)\longrightarrow \R^+$) whose value in point $\eps\in (0,1)$ is defined as the infimum of those $\delta>0$ such that for every $(x,T)\in B_X\times B_{L(X,Y)}$ ($(x,T)\in S_X\times S_{L(X,Y)}$ respectively) with $ \norm{T(x)}>1-\eps$, there is $(z,F)\in S_X\times S_{L(X,Y)}$ with $\|x-z\|<\delta$ and $\|T-F\|<\delta$.
\end{definition}

Under the notation
\begin{align*}
&\Pi_\eps(X,Y) \,= \left\{(x,T)\in X\times L(X,Y) : \|x\|\leq1, \|T\|\leq1, \, \|T(x)\| > 1-\varepsilon \right\},\\
&\Pi^S_\varepsilon(X,Y) = \left\{(x,T)\in X\times L(X,Y) : \|x\|=\|T\|=1, \, \|T(x)\| > 1-\varepsilon \right\}, \\
&\Pi(X,Y) \,\,\,= \left\{(x,T)\in X\times L(X,Y) : \|x\|=1, \|T\|=1, \, \|T(x)\|= 1 \right\},
\end{align*}
the definition can be rewritten as follows:
\begin{align*}
\Phi(X,Y,\varepsilon)=\underset{(x,T)\in\Pi_\varepsilon(X,Y)}{\sup}\quad\underset{(z,F)\in\Pi(X,Y)}{\inf} \quad \max\{\|x-z\|, \|T-F\|\},
\end{align*} 
\begin{align*}
\Phi^S(X,Y,\varepsilon)=\underset{(x,T)\in\Pi^S_\varepsilon(X,Y)}{\sup}\quad\underset{(z,F)\in\Pi(X,Y)}{\inf} \quad \max\{\|x-z\|, \|T-F\|\}.
\end{align*} 
Evidently, $\Phi^S(X,Y,\eps)\leq\Phi(X,Y,\eps)$, so any estimation from above for $\Phi(X,Y,\cdot)$ is also valid for $\Phi^S(X,Y,\cdot)$ and any estimation from below for $\Phi^S(X,Y,\cdot)$ is applicable to $\Phi(X,Y,\cdot)$. Also the following result is immediate.
\begin{remark} Let $X, Y$ be  Banach spaces, $\eps_1, \eps_2 > 0$ with $\eps_1<\eps_2$. Then
$\Pi_{\varepsilon_1}(X,Y)\subset \Pi_{\varepsilon_2}(X,Y)$ and $\Pi^S_{\varepsilon_1}(X,Y)\subset \Pi^S_{\varepsilon_2}(X,Y)$.
Therefore, $\Phi(X,Y,\varepsilon)$  and $\Phi^S(X,Y,\varepsilon)$ do not decrease as $\eps$ increases.
\end{remark}
Notice that a couple $(X,Y)$ has the Bishop-Phelps-Bollob\'{a}s property for operators if and only if $\Phi(X,Y,\varepsilon)\xrightarrow[\eps\rightarrow 0]{ }0$.

The aim of our paper is to estimate the Bishop-Phelps-Bollob\'{a}s modulus for operators which act to a Banach space with the property $\beta$. This paper is organized as follows. After the Introduction, in Section \ref{sec2} we will provide an estimation from above for $\Phi(X,Y,\varepsilon)$ for $Y$ possessing the property $\beta$ of Lindenstrauss (Theorem \ref{th:est_beta}) and an improvement for the case of $X$ being uniformly non-square (Theorem \ref{th:est_beta_non_square}). Section \ref{sec3} is devoted to estimations of $\Phi(X,Y,\varepsilon)$ from below and related problems. As a bi-product of these  estimations we obtain an interesting effect (Theorem \ref{theor-non-contin})  that $\Phi(X,Y,\varepsilon)$ is not continuous with respect to the variable $Y$. In  Section \ref{sec4} we consider a modification of  the above moduli which appear if one approximates by pairs $(y, F)$ with $\|F\|= \|Fy\|$ without requiring $\|F\| = 1$. Finally, in a very short Section \ref{sec5} we speak about a natural question which we did not succeed to solve.


\section{Estimation from above}\label{sec2}

Our first result is the upper bound of the Bishop-Phelps-Bollob\'{a}s moduli for the case when the range space has the property $\beta$ of Lindenstrauss.

\begin{theorem}\label{th:est_beta}
Let $X$ and $Y$ be Banach spaces such that $\beta(Y)\leq\rho$. Then for every $\varepsilon\in(0, 1)$
\begin{align}\label{eq_beta_main}
\Phi^S(X,Y,\varepsilon)\leq\Phi(X,Y,\varepsilon)\leq \min\left\{ \sqrt{2\varepsilon}\sqrt{\frac{1+\rho}{1-\rho}},\,\, 2\right\}.
\end{align}

\end{theorem}

The above result is a quantification of  \cite[Theorem 2.2]{Acosta2008} which states that if $Y$ has the property $\beta$, then for any Banach space $X$ the pair $(X, Y)$ has the Bishop-Phelps-Bollob\'{a}s property for operators. The construction is borrowed from the demonstration of \cite[Theorem 2.2]{Acosta2008}, but in order to obtain \eqref{eq_beta_main} we have to take care about details and need some additional work. At first, we have to modify a little bit the original results of Phelps about approximation of a functional $x^*$ and a vector $x$. 
 \begin{prop}[\cite{Phelps}, Corollary~2.2] \label{Corollary2.2Phelps} Let $X$ be a real Banach space, $x\in B_X$, $x^*\in S_{X^*}$, $\eta > 0$ and $x^*(x)> 1-\eta$. Then for any $k\in (0,1)$ there exist $\zeta^*\in X^*$ and $y\in S_X$ such that 
\begin{equation*} 
 \zeta^*(y) = \norm{\zeta^*},\qquad \norm{x-y} < \frac{\eta}{k}, \qquad \norm{x^*-\zeta^*} <  k.
\end{equation*}
\end{prop}

For our purposes we need an improvement which allows to take $x^*\in B_{X^*}$. 

\begin{lemma} \label{LemmaBPB} Let $X$ be a real Banach space, $x\in B_X$, $x^*\in B_{X^*}$, $\eps > 0$ and $x^*(x)> 1-\eps$. Then for any $k\in (0,1)$ there exist $y^*\in X^*$ and $z\in S_X$ such that 
\begin{equation}\label{eqBPB1}
 y^*(z) = \norm{y^*},\qquad \norm{x-z} < \frac{1-\frac{1-\eps}{\norm{x^*}}}{k}, \qquad \norm{x^*-y^*} <  k \norm{x^*}.
\end{equation}\label{eqBPB2}
Moreover, for any $\tilde{k}\in [\eps/2, 1)$ there exist $z^*\in S_{X^*}$ and $z\in S_X$ such that 
\begin{equation} 
 z^*(z) = 1,\qquad \norm{x-z} < \frac{\eps}{\tilde{k}}, \qquad \norm{x^*-{z}^*} <  2\tilde{k}.
\end{equation}
\end{lemma}
\begin{proof}
We have that $\frac{x^*}{\norm{x^*}}(x)> 1-\eta$ for $\eta=1-\frac{1-\eps}{\norm{x^*}}$ and we can apply Proposition \ref{Corollary2.2Phelps}. So, for any $k\in (0,1)$ there exist $\zeta^*\in X^*$ and $z\in S_X$ such that 
\begin{equation*} 
 \zeta^*(z) = \norm{\zeta^*},\qquad \norm{x-z} < \frac{\eta}{k}, \qquad \norm{\frac{x^*}{\norm{x^*}}-\zeta^*} <  k.
\end{equation*}
In order to get \eqref{eqBPB1} it remains to introduce $y^*=\norm{x^*}\cdot\zeta^*$. This functional also attains its norm at $z$ and 
\begin{equation*} 
\norm{x^*-{y^*}}= \norm{x^*}\cdot\norm{\frac{x^*}{\norm{x^*}}-\zeta^*}<  k \norm{x^*}.
\end{equation*}
In order to demonstrate  the ``moreover'' part,  take
$$ 
k=\frac{\tilde{k}(\norm{x^*}-(1-\eps))}{\eps \norm{x^*}}.
$$
The inequality $\norm{x^*} \ge x^*(x) > 1-\eps$ implies that $k > 0$. On the other hand, $k=\tilde{k}\left(\frac{1}{\eps} - \frac{(1-\eps)}{\eps  \norm{x^*}}\right)\le \tilde{k}\left(\frac{1}{\eps} - \frac{(1-\eps)}{\eps}\right) = \tilde{k} < 1$, so for this $k$ we can find  $y^*\in X^*$ and $z\in S_X$ such that \eqref{eqBPB1} holds true. Denote $z^*=\frac{y^*}{\norm{y^*}}$.
Then $\norm{x-z} < \eps/\tilde{k}$ and
\begin{align*}
\norm{x^*-z^*} &\leq \norm{x^*-y^*} +\norm{y^*-z^*}\leq \norm{x^*-y^*} + |1-\norm{y^*}| \\ &\leq \norm{x^*-y^*}+|1-\norm{x^*}+\norm{x^*}-\norm{y^*}|\leq 2 \norm{x^*-y^*}+1 - \norm{x^*}.
\end{align*}
So, we have
\begin{align*}
\norm{x^*-z^*}<2k\norm{x^*}+1 - \norm{x^*}=\frac{2\tilde{k}\cdot (\norm{x^*}-(1-\eps))}{\eps}+1 - \norm{x^*}\leq 2\tilde{k}.
\end{align*}
The last inequality holds, since the function $f(t)=\frac{2\tilde{k}\cdot (t-(1-\eps))}{\eps}+1 - t$ with $t\in (1-\eps, 1)$, is increasing when $\tilde{k}\geq \eps/2$, so $\max f = f(1)=2\tilde{k}$.
\end{proof}

\begin{remark}\label{rem_attain}
One can easily see that for $\tilde k < \frac{\eps}{2}$ the ``moreover'' part with \eqref{eqBPB2} is trivially true (and is not sharp) because in this case the inequality $\norm{x-{z}} \le \frac{\eps}{\tilde k}$ is weaker than the triangle inequality  $\norm{x-{z}} \le 2$, so one can just use the density of the set of norm-attaining functionals in order to get the desired  $(z,z^*)\in \Pi(X)$ with $\norm{x^*-{z}^*} <  2\tilde{k}$. 
\end{remark}

\begin{proof}[Proof of Theorem \ref{th:est_beta}]
We will use the notations $\{y_\alpha: \alpha\in \Lambda\}\subset S_Y$ and $\{y^*_\alpha: \alpha\in  \Lambda\}\subset S_Y^*$ from Definition \ref{def:beta} of  the property $\beta$.  

Consider $T\in B_{L(X,Y)}$ and $x\in B_X$ such that $\|Tx\|> 1-\eps$. 
According to (iii) of Definition \ref{def:beta}, there is $\alpha_0\in \Lambda$ such that $|y^*_{\alpha_0}(Tx)|> 1-\eps$. By Lemma \ref{LemmaBPB}, for
any  $k \in [\frac{\eps}{2}, 1)$ and for any $\delta>0$ there exist $z^*\in S_{X^*}$ and $z\in S_X$ such that $|z^*(z)|=1$, $\|z-x\|< \eps/k$ and $\|z^*-T^*(y^*_{\alpha_0})\|< 2k$.

For $\eta = 2k\frac{\rho}{1-\rho}$ let us introduce the following operator $S\in L(X,Y)$
\begin{align}\label{form_oper}
S(v) = T(v)+[(1+\eta)z^*(v)-(T^*y^*_{\alpha_0})(v)]y_{\alpha_0}.
\end{align}

Remark, that for all $y^*\in Y^*$
$$
S^*(y^*) = T^*(y^*)+[(1+\eta)z^* -T^*y^*_{\alpha_0}]y^*(y_{\alpha_0}).
$$
According to (iii) of Definition \ref{def:beta} the set $\{y^*_{\alpha}: \alpha \in \Lambda\}$ is norming for $Y$, consequently $\norm{S} = \sup_\alpha \norm{S^*{y^*_{\alpha}}}$.  Let us calculate the norm of $S$. 
$$\norm{S} \geq \norm{S^*(y^*_{\alpha_0})} = (1+\eta)\norm{z^*} = 1+\eta.$$
On the other hand for $\alpha \neq \alpha_0$ we obtain
$$
\norm{S^*(y^*_{\alpha})}\leq 1+\rho (\norm{z^* - T^*(y^*_{\alpha_0})} + \eta\norm{z^*})< 1+\rho(2k+\eta) = 1+\eta.
$$
Therefore, 
$$
\norm{S} = \norm{S^*(y^*_{\alpha_0})} = (1+\eta)\norm{z^*}=|y^*_{\alpha_0}(S(z))|\leq \norm{S(z)}\leq\norm{S}.
$$
So, we have $\|S\|=\|S(z)\|=1+\eta$. Also,
 $\|S-T\| \le \eta + \|z^*-T^*(y^*_{\alpha_0})\| < \eta + 2k$.

Define $F:=\frac{S}{\|S\|}$. Then $\|F\|=\|F(z)\|=1$ and $\|S - F\|=\|S\|(1-\frac{1}{1+\eta})=\eta$. So, $\|T-F\|< 2k+2\eta$. 

Therefore, we have that
\begin{align*}
\|z-x\|< \eps/k 	\mbox{ and } \|T-F\|< 2k \frac{1+\rho}{1-\rho}.
\end{align*}

Let us substitute $k=\sqrt{\frac{\varepsilon}{2}\cdot \frac{1-\rho}{1+\rho}}\,$  $\left( \text{here we need\,} \eps\leq\frac{2(1-\rho)}{1+\rho} \text{\,to have\,} k\in[\eps/2,1) \right)$. Then we obtain 
\begin{align*}
\max\{\|z-x\|, \|T-F\|\}< \sqrt{2\varepsilon}\sqrt{\frac{1+\rho}{1-\rho}}.
\end{align*}
Finally, if $\eps>\frac{2(1-\rho)}{1+\rho}$, we can use the triangle inequality to get the evident estimate $\max\{\|z-x\|, \|T-F\|\} \le 2$.
\end{proof}


Our next goal is to give an improvement for a uniformly non-square domain space $X$. We recall that uniformly non-square spaces were introduced by James \cite{James} as those spaces whose two-dimensional subspaces are uniformly separated (in the sense of Banach-Mazur distance) from $\ell_1^{(2)}$. A Banach space $X$ is uniformly non-square if and only if there is $\alpha>0$ such that
$$
\frac{1}{2}(\|x+y\|+\|x-y\|)\leq 2-\alpha
$$
for all $x, y \in B_X$. The \emph{parameter of uniform non-squareness} of $X$, which we denote $\alpha(X)$, is the best possible value of $\alpha$ in the above inequality. In other words,
$$
\alpha(X) := 2-\underset{x,y\in B_X}{\sup}\left\lbrace\frac{1}{2}(\|x+y\|+\|x-y\|)\right\rbrace.
$$

In \cite[Theorem~3.3]{Modul2015} it was proved that for a uniformly non-square space $X$ with the parameter of uniform non-squareness $\alpha(X)>\alpha_0 > 0$ 
\begin{equation*}
\Phi^S_X(\eps) \leq \sqrt{2\eps}\,\sqrt{1-\frac{1}{3}\alpha_0} \quad \text{for}\quad \eps \in \left(0, \frac{1}{2} - \frac{1}{6}\alpha_0\right).
\end{equation*}
To obtain this fact the authors proved the following technical result.
\begin{lemma}\label{main-lemma}
Let $X$ be a Banach space with $\alpha(X)> \alpha_0$. Then for every $x\in S_X, y\in X$ and  every $ k \in (0, \frac12]\,$ if  $\,\|x-y\|\leq k\,$  then
$$
\left\|x-\frac{y}{\|y\|}\right\|\leq 2k\left(1- \frac{1}{3}\alpha_0\right).
$$
\end{lemma}
\begin{theorem}\label{th:est_beta_non_square}
Let $X$ and $Y$ be Banach spaces such that $\beta(Y)\leq\rho$, $X$ is uniformly non-square with $\alpha(X)> \alpha_0$, and $\eps_0=\min\left\{\frac{2}{(1-1/3\alpha_0)}\frac{1-\rho}{1+\rho}, \frac{1}{2}\frac{1+\rho}{1-\rho}(1- 1/3\alpha_0)\right\}$. Then  for any $0<\varepsilon<\eps_0 $ 
\begin{align}\label{eq_beta_main_non_square}
\Phi^S(X,Y,\varepsilon)\leq  \sqrt{2\varepsilon\left(1- \frac{1}{3}\alpha_0\right)}\sqrt{\frac{1+\rho}{1-\rho}}\,\,.
\end{align}

\end{theorem}

Before proving the theorem, we need a preliminary result.

\begin{lemma}\label{main_non-sq_lemma}
Let $X$ be a Banach space with $\alpha(X)> \alpha_0$. Then for every $0<\eps<1$ and for every $(x,x^*)\in S_X\times B_{X^*}$ with $x^*(x)> 1-\eps$, and for every $ k \in [\frac{\eps}{2(1-1/3\alpha_0)}, \frac12]$ there is $(y,y^*)\in \Pi(X)$ such that
\begin{align*}
\|x-y\|<\frac{\eps}{k}\quad \emph{and} \quad
\|x^*-y^*\|< 2k\left(1- \frac{1}{3}\alpha_0\right).
\end{align*}
\end{lemma}
\begin{proof}
The  reasoning is almost  the same as in Lemma \ref{LemmaBPB}. 
We have that $\frac{x^*}{\norm{x^*}}(x)> 1-\eta$ for $\eta=1-\frac{1-\eps}{\norm{x^*}}$ and we can apply Proposition \ref{Corollary2.2Phelps} for every $\tilde{k}\in (0,1/2]$. Let us take 
$$\tilde{k}=\frac{ k (\norm{x^*} - (1-\eps))}{\eps \norm{x^*}}.$$
The inequality $\norm{x^*} \ge x^*(x) > 1-\eps$ implies that $\tilde{k} > 0$. On the other hand, $\tilde{k}=k\left(\frac{1}{\eps} - \frac{(1-\eps)}{\eps  \norm{x^*}}\right)\le k\left(\frac{1}{\eps} - \frac{(1-\eps)}{\eps}\right) = k < 1/2$, so for this $\tilde{k}$ we can find  $\zeta^*\in X^*$ and $z\in S_X$  such that 
\begin{equation*} 
 \zeta^*(z) = \norm{\zeta^*},\qquad \norm{x-z} < \frac{\eta}{\tilde{k}}, \qquad \norm{\frac{x^*}{\norm{x^*}}-\zeta^*} <  \tilde{k}.
\end{equation*}

Consider $z^*=\frac{\zeta^*}{\norm{\zeta^*}}$. According to Lemma \ref{main-lemma} 
\begin{align*}
\norm{\frac{x^*}{\norm{x^*}}-z^*} <  2\tilde{k}\left(1- \frac{1}{3}\alpha_0\right).
\end{align*}
Then $\norm{x-z} < \eps/k$ and
\begin{align*}
\|x^*-z^*\|&=\norm{x^*}\cdot\left\|\frac{x^*}{\norm{x^*}}-\frac{z^*}{\norm{x^*}}\right\| \leq \norm{x^*}\left(\left\|\frac{x^*}{\norm{x^*}}-z^*\right\|+\left\|z^* - \frac{z^*}{\norm{x^*}}  \right\|\right)\\ 
&=\norm{x^*}\left( 2\tilde{k}(1- 1/3\alpha_0)+\left|1-\frac{1}{\norm{x^*}}\right|\right) \\
&=2\frac{k(\norm{x^*}-(1-\eps))}{\eps}(1- 1/3\alpha_0)+1-\norm{x^*} \leq 2k(1- 1/3\alpha_0).
\end{align*}
The last inequality holds, because if we consider the function $$f(t)=\frac{2k(1- 1/3\alpha_0)\cdot (t-(1-\eps))}{\eps}+1 - t$$ with $t\in (1-\eps, 1]$, then ${f'}\geq 0$ if $k\geq \frac{\eps}{2(1- 1/3\alpha_0)}$, so $\max f = f(1)=2k\left(1- \frac{1}{3}\alpha_0\right)$.
\end{proof}

\begin{proof}[Proof of Theorem \ref{th:est_beta_non_square}]
The proof is a minor modification of the one given for Theorem \ref{th:est_beta}. 

In order to get \eqref{eq_beta_main_non_square} for $\eps<\eps_0$ we
consider $T\in S_{L(X,Y)}$ and $x\in S_X$ such that $\|T(x)\|> 1-\eps$. 
Since $Y$ has the property $\beta$, there is $\alpha_0\in \Lambda$ such that $|y^*_{\alpha_0}(T(x))|> 1-\eps$. By Lemma \ref{main_non-sq_lemma}, for
any  $k \in [\frac{\eps}{2(1-1/3\alpha_0)}, \frac12]$ and for any $\eps>0$ there exist $z^*\in S_{X^*}$ and $z\in S_X$ such that $|z^*(z)|=1$, $\|z-x\|< \eps/k$ and $\|z^*-T^*(y^*_{\alpha_0})\|< 2k(1- 1/3\alpha_0)$.

For $\eta = 2k(1- 1/3\alpha_0)\frac{\rho}{1-\rho}$ we define $S\in L(X,Y)$ by the formula \eqref{form_oper} and take $F:=\frac{S}{\norm{S}}$.
By the same argumentation as before, we have that
\begin{align*}
\|x-z\|< \eps/k 	\mbox{ and } \|T-F\|< 2k\left(1- \frac{1}{3}\alpha_0\right) \frac{1+\rho}{1-\rho}.
\end{align*}
Let us substitute $k=\sqrt{\frac{\varepsilon}{2(1- 1/3\alpha_0)}\cdot \frac{1-\rho}{1+\rho}}$ (here we need $\eps<\eps_0$). Then we obtain that 
\begin{align*}
\max\{\|z-x\|, \|T-F\|\}< \sqrt{2\varepsilon\left(1- \frac{1}{3}\alpha_0\right)}\sqrt{\frac{1+\rho}{1-\rho}}.
\end{align*}
\end{proof}


\section{Estimation from below}\label{sec3}
\subsection{Improvement for $\Phi(\ell^{(2)}_1,Y,\varepsilon)$}
We tried our best, but unfortunately we could not find an example demonstrating the sharpness of \eqref{eq_beta_main} in Theorem \ref{th:est_beta}. So, our goal is less ambitious. We are going to present examples of pairs $(X, Y)$  in which the estimation of $\Phi(X,Y, \varepsilon)$ from below is reasonably close to the estimation from above given in \eqref{eq_beta_main}.

Theorem \ref{th:est_beta_non_square} shows that in order to check the sharpness of Theorem \ref{th:est_beta} one has to  try those domain spaces $X$ that are not uniformly non-square. The simplest of them is  $X=\ell^{(2)}_1$. In \cite[Example 2.5]{Modul2014} this space worked perfectly for the Bishop-Phelps-Bollob\'{a}s modulus for functionals. Nevertheless, this is not so when one deals with the Bishop-Phelps-Bollob\'{a}s modulus for operators. Namely, the following theorem demonstrates that for $X=\ell^{(2)}_1$ the estimation given in Theorem \ref{th:est_beta} can be improved.

\begin{theorem}\label{th:ell_1}
Let $Y$ be Banach spaces and  $\beta(Y)\leq\rho$. Then 
\begin{align}\label{eq:ell_1}
\Phi^S(\ell^{(2)}_1,Y,\varepsilon)\leq\Phi(\ell^{(2)}_1,Y,\varepsilon)\leq\min\left\{\sqrt{2\varepsilon}\frac{1+\rho}{\sqrt{1-\rho^2+\frac{\eps}{2}\rho^2}+\rho\sqrt{\frac{\eps}{2}}}, 1\right\}.
\end{align}
\end{theorem}
To prove this theorem we need a preliminary result.

\begin{lemma}\label{lemma_grani}
Let $Y$ be Banach space such that $\beta(Y)\leq\rho$, $y\in B_Y$, $\{y_\alpha: \alpha\in \Lambda\}\subset S_Y$, $\{y^*_\alpha: \alpha\in  \Lambda\}\subset S_Y^*$  be the sets from Definition \ref{def:beta}.  For given $r \in (0,1)$, $\alpha_0\in  \Lambda$ suppose that  $y^*_{\alpha_0}(y)\geq 1- r$. Then there is $z\in S_Y$ such that
\begin{enumerate}
\item $y^*_{\alpha_0}(z)=1$;

\item $|y^*_{\alpha}(z)| \leq 1$ for all $\alpha\in  \Lambda$;

\item $\norm{y-z}\leq\frac{r(1+\rho)}{1-\rho+\rho r}$.
\end{enumerate}
\end{lemma}
\begin{proof}
Suppose that $y^*_{\alpha_0}(y)=1-r_0, r_0\in[0, r]$. According to (i) of   Definition \ref{def:beta} $y^*_{\alpha_0}(y_{\alpha_0})=1$.
Let us check  the properties (i)-(iii) for
\begin{align*}
z:=\frac{r_0}{1-\rho+\rho r_0}y_{\alpha_0}+\left(1-\frac{r_0\rho}{1-\rho+\rho r_0}\right)y.
\end{align*}

\noindent (i) $y^*_{\alpha_0}(z) = \frac{r_0}{1-\rho+\rho r_0}+\left(1-\frac{r_0\rho}{1-\rho+\rho r_0}\right)(1-r_0)=1$;

\noindent (ii) For every $\alpha\neq\alpha_0$ we have $|y^*_{\alpha}(z)|\leq\frac{r_0}{1-\rho+\rho r_0}\cdot \rho+\left(1-\frac{r_0\rho}{1-\rho+\rho r_0}\right)=1$;

\noindent (iii) As $\{y^*_\alpha: \alpha\in  \Lambda\}\subset S_Y^*$ is a 1-norming subset, so $\norm{y-z}=\underset{\alpha\in  \Lambda}{\sup} |y^*_\alpha(y-z)|$. Notice that $|y^*_{\alpha_0}(y-z)| \leq r$, and for every $\alpha\neq\alpha_0$ we have 
\begin{align*}
|y^*_{\alpha}(y-z)|=\left|\frac{r_0}{1-\rho+\rho r_0}y^*_{\alpha}(y)-\frac{r_0}{1-\rho+\rho r_0}y^*_{\alpha}(y_{\alpha_0})\right|\leq\frac{r_0(1+\rho)}{1-\rho+\rho r_0}\leq\frac{r(1+\rho)}{1-\rho+\rho r}.
\end{align*}
So, $\norm{y-z}\leq \max\left\{r, \frac{r(1+\rho)}{1-\rho+\rho r}\right\}=\frac{r(1+\rho)}{1-\rho+\rho r}$.

\noindent Finally, (i) and (ii) imply that  $z\in S_Y$.
\end{proof}

\begin{remark}\label{rem_attain}
For every operator $T \in L(\ell_1^{(2)}, Y)$
$$\norm{T}=\max\{\norm{T(e_1)}, \norm{T(e_1)}\}.$$
Moreover, if the operator $T\in L(\ell_1^{(2)}, Y)$ attains its norm in some point $x\in S_{\ell_1^{(2)}}$
which does not coincide neither with $\pm e_1$, nor with $\pm e_2$, then either the segment $[T(e_1), T(e_2)]$, or $[T(e_1), - T(e_2)]$ has to lie on the sphere $\|T\|S_Y$. 
\end{remark}

\begin{proof}[Proof of Theorem \ref{th:ell_1}]
Let us denote $A(\rho, \eps):=\sqrt{2\varepsilon}\frac{1+\rho}{\sqrt{1-\rho^2+\frac{\eps}{2}\rho^2}+\rho\sqrt{\frac{\eps}{2}}}$. Notice that $A(\rho, \eps)$ is increasing as a function of $\rho$, in particular  $\sqrt{2\varepsilon}=A(0, \eps) \leq A(\rho, \eps) \leq A(1, \eps) = 2$.
 
%
%
%
%
%
%
%
 
We are going to demonstrate that for every pair $(x, T)\in \Pi_\eps(\ell^{(2)}_1, Y)$ 
there exists a pair $(y, F)\in\Pi(\ell^{(2)}_1, Y)$ with
\begin{align*}
\max\{\norm{x-y}, \norm{T-F}\}\leq \min\{A(\rho, \eps), 1\}.
\end{align*}
Without loss of generality suppose that  $x=(t(1-\delta), t\delta), \delta\in [0, 1/2], t\in [1-\eps, 1]$. Evidently, $\|x\|=t$.  First, we make sure that $\Phi(\ell^{(2)}_1,Y,\varepsilon)\leq 1.$ Indeed, we can always approximate $(x, T)$ by the pair $y:=e_1$ and $F$ determined by formula $F(e_i):=T(e_i)/\norm{T(e_i)}$. Then $\norm{x-e_1}=2t\delta+1-t\leq 1$ and $\norm{T-F}\leq 1$.

It remains to show that $\Phi(\ell^{(2)}_1,Y,\varepsilon)\leq A(\rho, \eps),$ when $A(\rho, \eps)< 1$. As  $A(\rho, \eps)\geq\sqrt{2\eps}$ we must consider $\eps\in(0, 1/2)$.  Since $Y$ has the property $\beta$, we can select an $\alpha_0$ such that $|y^*_{\alpha_0}(T(x))|> 1-\eps$. Without loss of generality we can assume $y^*_{\alpha_0}(T(x))> 1-\eps$. Then  $y^*_{\alpha_0}\left(T\left(\frac{x}{t}\right)\right)> 1-\eps'$, where $\eps'=\frac{\eps-(1-t)}{t}\in(0, \eps)$. 
Therefore, 
\begin{align}\label{eq_gran}
y^*_{\alpha_0}(T(e_1))> 1-\frac{\eps'}{1-\delta} \quad 
\text{and}\quad y^*_{\alpha_0}(T(e_2)) > 1-\frac{\eps'}{\delta}.
\end{align}
 We are searching for an approximation of $(x, T)$ by a pair $(y, F)\in \Pi(\ell_1^{(2)}, Y)$. Let us consider two cases:

Case I:  \, $2t\delta+1-t \le A(\rho, \eps)$. In this case we approximate $(x, T)$ by the vector $y:=e_1$ and the operator $F$ such that $F(e_1):=\frac{T(e_1)}{\norm{T(e_1)}}, F(e_2):=T(e_2)$. Then
\begin{align*}
\norm{x-y}\leq 2t\delta+1-t \le A(\rho, \eps)\,\, \text{and}\,\, \norm{T-F}\leq 1- \norm{T(e_1)}\leq \frac{\eps}{1-\delta}\leq 2\eps\leq A(\rho, \eps).
\end{align*}

Case II:  \, $2t\delta+1-t > A(\rho, \eps)$. Remark, that in this case 
$2t\delta+1-t >  \sqrt{2\eps}$, and consequently (here we use that  $A(\rho, \eps)\geq\sqrt{2\eps}$,  $\eps\in(0, 1/2)$ and $t \in (0, 1]$),
$$
\delta > \frac{\sqrt{2\eps} - (1 - t)}{2t} \ge \eps'.
$$
 According to \eqref{eq_gran} we can apply Lemma \ref{lemma_grani} for the points $T(e_1)$ and $T(e_2)$ with $r=\frac{\eps'}{\delta} < 1$. 
So, there are $z_1, z_2\in S_Y$ such that $y^*_{\alpha_0}(z_1)=y^*_{\alpha_0}(z_2)=1$ and 
$$ \max\{\norm{T(e_1)-z_1}, \norm{T(e_2)-z_2}\}\leq \frac{\frac{\eps'}{\delta}(1+\rho)}{1-\rho +\rho \frac{\eps'}{\delta}}.$$
Denote $y:=x/t \in S_{\ell_1^{(2)}}$ and define $F$ as follows:
$$F(e_1):=z_1, \quad F(e_2):=z_2.$$
Then $\|F\| = 1$, $\|F(y)\|  \ge y^*_{\alpha_0}(Fy) = 1$, so $F$ attains its norm in  $y$ and 
\begin{align*}
\norm{T-F}\leq \frac{\frac{\eps'}{\delta}(1+\rho)}{1-\rho +\rho \frac{\eps'}{\delta}}.
\end{align*}
So, in this case
\begin{align*}
\norm{x-y}\leq \eps\leq A(\rho, \eps)\,\, \text{and}\,\, \norm{T-F}\leq \frac{(1+\rho)\frac{\eps-1+t}{t\delta}}{1-\rho+\rho\frac{\eps-1+t}{t\delta}}.
\end{align*}

To prove our statement we must show that if  $2t\delta+1-t > A(\rho, \eps)$, then  $\frac{(1+\rho)\frac{\eps-1+t}{t\delta}}{1-\rho+\rho\frac{\eps-1+t}{t\delta}} \le A(\rho, \eps)$.  Let us denote $f(t, \delta)=2t\delta+1-t$ and $g(t, \delta)=\frac{(1+\rho)\frac{\eps-1+t}{t\delta}}{1-\rho+\rho\frac{\eps-1+t}{t\delta}}=\frac{(1+\rho)(\eps-1+t)}{(1-\rho)t\delta+\rho(\eps-1+t)}$. So, we need to demonstrate that 
\begin{align}\label{minfg}
\min\{f(t, \delta), g(t, \delta)\}\leq A(\rho, \eps)\, \text{for all}\, \delta\in[0, 1/2]\, \text{ and for all}\, t\in [1-\eps, 1].
\end{align}

Notice that for every fixed $t\in [1-\eps, 1]$  the function $f(t, \delta)$ is increasing as $\delta$ increases  and $g(t, \delta)$ is decreasing as $\delta$ increases. 
So, if we find $\delta_0$ such that $f(t, \delta)=g(t, \delta)$, then $\min\{f(t, \delta), g(t, \delta)\}\leq f(t, \delta_0)$.
If we denote  $u= f(t, \delta)= 2t\delta+1-t$ the equation  $f(t, \delta)=g(t, \delta)$  transforms to
\begin{align}\label{eqfg2}
u=2-\frac{2(1-\rho)(u-\eps)}{(t-1+\eps)(1+\rho)+(u-\eps)(1-\rho)}.
\end{align}
The right-hand side of this equation is increasing as $t$ increases, so the positive solution of the equation \eqref{eqfg2} $u_t$  is also increasing. This means that we obtain the greatest possible solution, if we substitute $t=1$. Then we get the equation
\begin{align*}
u^2\frac{1+\rho}{2}+u \rho\eps-\eps(1+\rho)=0.
\end{align*}
From here $u=A(\rho, \eps)$, and so, the inequality \eqref{minfg} holds.

%
%
%
\end{proof}

\subsection{Estimation from below for $\Phi^S(\ell^{(2)}_1,Y,\eps)$} \label{ssectestbelow} 
So, if $X=\ell_1^{(2)}$, the estimation for  the Bishop-Phelps-Bollob\'{a}s modulus is somehow better than in Theorem \ref{th:est_beta}. Nevertheless, considering $\ell_1^{(2)}$ we can obtain some interesting estimations from below for $\Phi^S(\ell^{(2)}_1,Y,\eps)$. Notice that the estimations \eqref{eq_beta_main} and \eqref{eq:ell_1} give the same asymptotic behavior  when $\eps$ is convergent to $0$. Our next proposition gives the estimation for $\Phi^S(\ell^{(2)}_1,Y,\eps)$ from below, when $\beta(Y)=0$.

\begin{theorem}\label{th:est_beta2}
For every Banach space $Y$ 
\begin{align*}
\Phi^S(\ell^{(2)}_1,Y,\varepsilon)\geq\min\{\sqrt{2\varepsilon}, 1\}.
\end{align*}  
In particular, $\Phi^S(\ell^{(2)}_1,Y,\varepsilon) = \min\{\sqrt{2\varepsilon}, 1\}$ if  $\beta(Y)=0$.
\end{theorem}
\begin{proof}

To prove our statement we must show that $\Phi^S(\ell^{(2)}_1,Y,\varepsilon)\geq \sqrt{2\varepsilon}$ for $\eps\in (0,1/2)$. The remaining inequality $\Phi^S(\ell^{(2)}_1,Y,\varepsilon)\geq 1$ for  $\eps > 1/2$ will follow from the monotonicity of $\Phi^S(\ell^{(2)}_1,Y,\cdot)$. So, for every $\eps\in (0,1/2)$ 
and for every $\delta > 0$ we are looking for a pair $(x, T) \in \Pi^S_\eps(\ell^{(2)}_1, Y)$ such that 
\begin{equation*}
\max\left\{ \norm{x-y}, \norm{T-F}\right\}\geq \sqrt{2\varepsilon} - \delta.
\end{equation*}
for every pair $(y, F) \in \Pi(\ell^{(2)}_1, Y)$.  Fix $\xi\in S_Y$  and $\eps_0 < \eps$ such that $\sqrt{2\eps_0} > \sqrt{2\varepsilon} - \delta$. Consider the following operator $T  \in S_{L(\ell_1^{(2)},Y)}$:
$$
T(z_1, z_2)=(z_1+(1-\sqrt{2\eps_0})z_2)\xi
$$ 
and take $x=(1-\sqrt{\eps_0/2},\sqrt{\eps_0/2})\in S_{\ell_1^{(2)}}$. Then $ \norm{T(x)} = 1-\eps_0 > 1 - \eps$.  
To approximate the pair $(x,T)$  by a pair $(y, F)\in \Pi(\ell^{(2)}_1, Y)$ we have two possibilities: either $y$ is an extreme point of $B_{\ell^{(2)}_1}$ or $F$ attains its norm in a point that belongs to $\conv\{e_1, e_2\}$, and so attains its norm in both points $e_1, e_2$. In the first case we are forced to have $y=(1,0)$, and then $\|x-y\|=\sqrt{2\eps_0}> \sqrt{2\varepsilon} - \delta$. In the second case we have
$\|F-T\|\geq\|F(e_2)-T(e_2)\|\geq \|F(e_2)\|-\|T(e_2)\|=\sqrt{2\eps_0} > \sqrt{2\varepsilon} - \delta$.
\end{proof}

Our next goal is to estimate the spherical Bishop-Phelps-Bollob\'{a}s modulus from below for the values of parameter $\rho$ between $1/2$ and 1. Fix a $\rho \in [\frac{1}{2},  1)$ and denote $Y_\rho$ the linear space $\R^2$ equipped with the norm 
\begin{equation}\label{eq_norm}
\|x\|_\rho= \max\left\lbrace | x_1  + \left(2-\frac{1}{\rho}\right) x_2| , | x_2 +  \left(2-\frac{1}{\rho}\right) x_1|, | x_1- x_2|\right\rbrace.
\end{equation}
In other words, 
$$\|(x_1,x_2)\| =\begin{cases}
| x_1- x_2|, &\text{if \,} x_1x_2\leq 0;\\
|  x_1  +  \left(2-\frac{1}{\rho}\right) x_2|,&\text{if \,} x_1x_2>0 \,\, \text{and \,} | x_1|>| x_2|;\\
| x_2 +   \left(2-\frac{1}{\rho}\right) x_1|,&\text{if \,}x_1x_2>0  \,\, \text{and \,}  | x_1|\leq | x_2|.
\end{cases}$$
and the unit ball $B_{\rho}$ of $X_\rho$ is the hexagon $absdef$, where 
$a = (1, 0); b = (\frac{\rho}{3\rho-1},\frac{\rho}{3\rho-1}); c = (0, 1); d = (-1, 0); e = (-\frac{\rho}{3\rho-1}, \frac{\rho}{3\rho-1})$; and $f = (0, - 1)$.

The dual space to $Y_\rho$ is $\R^2$ equipped with the polar to $B_{\rho}$ as its unit ball. So, the norm  on $Y^*=Y_\rho^*$ is given by the formula 
\begin{align*}
\|x\|_\rho^*=\|(x_1,x_2)\|^{*}=
\max\left\{| x_1|, | x_2|, \frac{\rho}{3\rho-1}| x_1 + x_2|\right\},
\end{align*}
and the unit ball $B_{\rho}^*$ of  $Y_\rho^*$ is the hexagon $a^*b^*c^*d^*e^*f^*$, where 
$a^* = (1, 2-\frac{1}{\rho}); b^* = \left(2-\frac{1}{\rho} , 1\right); c^* = (-1, 1); d^* = (-1, -\left(2-\frac{1}{\rho}\right)); e^* = (-\left(2-\frac{1}{\rho}\right), -1)$; and $f^* = (1, -1)$.
The corresponding spheres $S_{\rho}$ and $S_{\rho}^*$ are shown on Figures 1 and 2 respectively.
\\

\begin{picture}(300,250)

\put(0,125){\line(1,0){140}}
\put(75,40){\line(0,1){140}}
\put(140,125){\vector(1,0){20}}
\put(75,150){\vector(0,1){60}}
\put(17,67){\begin{tikzpicture}[scale=0.2]
\draw (0,0) --(10,10) -- (17,7) -- (20,0) -- (10,-10) -- (3,-7) -- (0,0);
\draw [fill] (0,0) circle [radius=0.3];;
\draw [fill] (10,10) circle [radius=0.3];;
\draw [fill] (17,7) circle [radius=0.3];;
\draw [fill] (20,0) circle [radius=0.3];;
\draw [fill] (10,-10) circle [radius=0.3];;
\draw [fill] (3,-7) circle [radius=0.3];;
\draw [fill] (5,5) circle [radius=0.3];;
\draw [fill] (13.5,8.5) circle [radius=0.3];;
\draw [fill] (18.5,3.5) circle [radius=0.3];;
\end{tikzpicture}}
\put(140,130){$a$}
\put(126,150){$y_1$}
\put(125,170){$b$}
\put(105,175){$y_2$}
\put(85,185){$c$}
\put(35,160){$y_3$}
\put(15,130){$d$}
\put(30,70){$e$}
\put(85,55){$f$}
\put(65,10){Figure 1}

\put(200,125){\line(1,0){140}}
\put(275,40){\line(0,1){140}}
\put(340,125){\vector(1,0){20}}
\put(275,150){\vector(0,1){60}}
\put(175,60){\begin{tikzpicture}[scale=0.2]
\draw (0,5) --(0,20) -- (15,20) -- (20,15) -- (20,0) -- (5,0) -- (0,5)node at (-2,5) {$d^*$}node at (-4,20) {$c^*=y^*_3$}node at (20,20) {$b^*=y^*_2$}node at (24,15) {$a^*=y^*_1$}node at (22,0) {$f^*$}node at (3,0) {$e^*$};
\draw [fill] (0,5) circle [radius=0.3];;
\draw [fill] (0,20) circle [radius=0.3];;
\draw [fill] (15,20) circle [radius=0.3];;
\draw [fill] (20,15) circle [radius=0.3];;
\draw [fill] (20,0) circle [radius=0.3];;
\draw [fill] (5,0) circle [radius=0.3];;
\end{tikzpicture}}
\put(250,10){Figure 2}

\end{picture}

\begin{prop} In the space $Y=Y_\rho$
$$\beta(Y)\leq\rho.$$
\end{prop}  

\begin{proof} 
Consider two sets: 

$$\left\{y_1 = \left(\frac{2\rho^2}{3\rho-1}, \frac{\rho-\rho^2}{3\rho-1}\right), y_2 = \left(\frac{\rho-\rho^2}{3\rho-1}, \frac{2\rho^2}{3\rho-1}\right), y_3 = \left(-\frac{1}{2}, \frac{1}{2}\right)\right\}\subset S_{Y_\rho}$$
and  $\{y_1^* = (1, 2-\frac{1}{\rho}), y_2^* = ( 2-\frac{1}{\rho}, 1), y_3^* = (-1, 1)\}\subset S_{Y_\rho^*}$.

Then   $\norm{y} = \sup\{|y_n^*(y)|: n=1,2,3\}$ for all $y\in Y_\rho$, $y^*_n(y_n) = 1$ for $n=1,2,3$ and
$|y^*_i(y_j)|\leq \rho$ for all $i \neq j$. Indeed,
$y^*_1(y_1) = \frac{2\rho^2+2\rho-2\rho^2-1+\rho}{3\rho-1}=1$; \quad
$y^*_1(y_2) = \frac{\rho-\rho^2+4\rho^2-2\rho}{3\rho-1}=\rho$; \quad
$y^*_1(y_3) = \frac{-1}{2}+1-\frac{1}{2\rho}=-\frac{1-\rho}{2\rho}\geq-\rho $, consequently $|y^*_1(y_3)| \leq \rho $ (here appears the restriction $\rho\geq1/2$); \quad
$y^*_2(y_1) = y^*_1(y_2)= \rho$; \quad
$y^*_2(y_2) = y^*_1(y_1)=1$; \quad
$y^*_2(y_3) = - y^*_1(y_3)\leq \rho$; \\
$|y^*_3(y_1)| = \left|\frac{-2\rho^2+\rho-\rho^2}{3\rho-1}\right|= \rho$; 
$y^*_3(y_2) = -y^*_3(y_1)=\rho$; and finally
$y^*_3(y_3) = \frac{1}{2}+\frac{1}{2}=1$.
\end{proof}

\begin{theorem}\label{th-hexagon}
 Let $ \rho \in [1/2, 1)$, $0<\eps< 1$. Then, in the space $Y = Y_\rho$ 
$$
\Phi^S(\ell_1^{(2)},Y,\varepsilon)\geq \min \left\lbrace\sqrt{\frac{2\rho\eps}{1-\rho}}, 1\right\rbrace.
$$
\end{theorem}

\begin{proof}
To prove our statement we must show that $\Phi^S(\ell_1^{(2)},Y,\varepsilon)\geq \sqrt{\frac{2\rho\eps}{1-\rho}}$ for $\eps\in\left(0,\frac{1-\rho}{2\rho}\right)$. The remaining inequality $\Phi^S(\ell^{(2)}_1,Y,\varepsilon)\geq 1$ for  $\eps \geq \frac{1-\rho}{2\rho}$ will follow from the monotonicity of $\Phi^S(\ell^{(2)}_1,Y,\cdot)$. So, for every $\eps\in \left(0,\frac{1-\rho}{2\rho}\right)$ and for every $\delta > 0$ we are looking for a pair $(x, T) \in \Pi^S_\eps(\ell^{(2)}_1, Y)$ such that 
\begin{equation*}
\max\left\{ \norm{x-y}, \norm{T-F}\right\}\geq \sqrt{\frac{2\rho\eps}{1-\rho}}  - \delta
\end{equation*}
for every pair $(y, F) \in \Pi(\ell^{(2)}_1, Y)$.

Fix an $\eps_0 < \eps$ such that $\sqrt{\frac{2\rho\eps_0}{1-\rho}} > \sqrt{\frac{2\rho\eps}{1-\rho}} - \delta$. Consider the point 
$$x=\left(1-\frac{\sqrt{2\rho\eps_0}}{2\sqrt{1-\rho}}, \frac{\sqrt{2\rho\eps_0}}{2\sqrt{1-\rho}} \right)\in S_{\ell_1^{(2)}}$$
 and $T\in L(\ell_1^{(2)},Y)$ such that 
 $$T(e_i)=\sqrt{\frac{2\rho\eps_0}{1-\rho}}e_i+\left(1-\sqrt{\frac{2\rho\eps_0}{1-\rho}}\right)\cdot b,$$ 
 where $b=\left(\frac{\rho}{3\rho-1}, \frac{\rho}{3\rho-1}\right)$ is the  extreme point of $S_Y$ from Figure 1.
 Notice that $\norm{T} = \norm{T(e_1)} = \norm{T(e_2)} = 1$ and $\|T(x)\|=1-\eps_0> 1-\eps$. 
 
The part of $S_{\ell_1^{(2)}}$ consisting of points that have a distance to $x$ less than or equal to $\sqrt{\frac{2\rho\eps_0}{1-\rho}}$ lies on the segment $[e_1, e_2)$. Consequently, in order to approximate the pair $(x,T)$ we have two options: to approximate the point $x$ by $e_1$, and then we can take $F:=T$; or as  $F$ choose an operator attaining its norm in some point of  $(e_1, e_2)$ (and hence in all points of  $[e_1, e_2]$), and then we can take $y:=x$.

In the first case we have $\norm{T-F}=0$ and $\norm{x-y}=\sqrt{\frac{2\rho\eps_0}{1-\rho}}> \sqrt{\frac{2\rho\eps}{1-\rho}} - \delta$. 
In the second case let us demonstrate that 
$$\|T-F\| = \max_i \|T(e_i)-F(e_i)\|\geq \sqrt{\frac{2\rho\eps_0}{1-\rho}}.$$
If it is not so, then for both values of $i = 1, 2$
$$
\|T(e_i)-F(e_i)\|  <  \sqrt{\frac{2\rho\eps_0}{1-\rho}} = \|T(e_i) - b\| .
$$
Since  $F$ attains its norm  in all points of  $[e_1, e_2]$, the line segment  $F([e_1, e_2])$ should lie on a line segment of $S_Y$, but the previous inequality makes this impossible, because $T(e_1)$ and $T(e_2)$ lie on different line segments of   $S_Y$ with $b$ being their only common point.
\end{proof}
\subsection{Non-continuity of the Bishop-Phelps-Bollob\'as modulus for operators}
It is known \cite[Theorem 3.3]{Modul2016} that both (usual and spherical) Bishop-Phelps-Bollob\'as moduli for functionals are continuous with respect to $X$.
As a consequence of Theorem \ref{th-hexagon} we will obtain that the  Bishop-Phelps-Bollob\'as moduli of a pair $(X, Y)$ as a function of $Y$ are not continuous in the sense of Banach-Mazur distance. 

Let $X$ and $Y$ be isomorphic. Recall that the Banach-Mazur distance between $X$ and $Y$ is the following quantity
$$d(X, Y) = \inf \{\norm{T}\norm{T^{-1}}:\quad T : X \rightarrow  Y  \text{ isomorphism}.\}$$ 

A sequence $Z_n$ of Banach spaces is said to be convergent to a Banach space $Z$ if $d(Z_n, Z) \underset{n \rightarrow \infty}{\longrightarrow }1.$

Notice, that $Y_\rho\underset{\rho\rightarrow 1}{\longrightarrow }\ell_1^{(2)}$.
\begin{theorem} \label{theor-non-contin}
Let $ \rho \in [1/2, 1)$ and $Y_\rho$ be the spaces defined in \eqref{eq_norm}. Then for every $\eps \in (0, \frac12)$ 
$$
\Phi(\ell_1^{(2)},Y_\rho,\varepsilon)\underset{\rho\rightarrow 1}{\centernot\longrightarrow}\Phi(\ell_1^{(2)},\ell_1^{(2)},\varepsilon), \, \mbox{ and } \, \Phi^S(\ell_1^{(2)},Y_\rho,\varepsilon)\underset{\rho\rightarrow 1}{\centernot\longrightarrow}\Phi^S(\ell_1^{(2)},\ell_1^{(2)},\varepsilon).
$$
\end{theorem}
\begin{proof}
On the one hand, from the Theorem \ref{th:est_beta} with $\rho=0$ we get for $\eps \in (0, \frac12)$ 
$$ \Phi^S(\ell_1^{(2)},\ell_1^{(2)},\varepsilon)\leq\Phi(\ell_1^{(2)},\ell_1^{(2)},\varepsilon)\leq \sqrt{2\eps} < 1.
$$
On the other hand, Theorem \ref{th-hexagon} gives $\Phi(\ell_1^{(2)},Y_\rho,\varepsilon)\geq \Phi^S(\ell_1^{(2)},Y_\rho,\varepsilon)\geq \min \left\lbrace\sqrt{\frac{2\rho\eps}{1-\rho}}, 1\right\rbrace \xrightarrow[\rho \to 1]{} 1$. 
\end{proof}

\subsection{Behavior of the $\Phi^S(X,Y,\varepsilon)$ when $\eps \rightarrow 0$}

In subsection \ref{ssectestbelow} using two-dimensional spaces $Y$ we were able to give the estimation only for $\rho\in [1/2, 1)$. This is not surprising, because in every $n$-dimensional Banach space with the property $\beta$ we have either $\rho=0$, or $\rho\geq\frac{1}{n}$. We did not find any mentioning of this in literature, so we give the proof of this fact.
\begin{prop}\label{prop_beta}
Let $Y^{(n)}$ be a Banach space of dimension $n$ with $\beta(Y^{(n)})\leq\rho< \frac{1}{n}$. Then $Y^{(n)}$ is isometric to $\ell^{(n)}_\infty$, i.e. $\beta(Y^{(n)})=0$.
\end{prop}

We need one preliminary result.
\begin{lemma}\label{prop_beta_lemma}
Let $Y^{(n)}$ be a Banach space of dimension $n$ with $\beta(Y^{(n)})\leq\rho< \frac{1}{n}$ and $\{y_\alpha: \alpha\in \Lambda\}\subset S_Y$, $\{y^*_\alpha: \alpha\in  \Lambda\}\subset S_Y^*$ be the sets from Definition \ref{def:beta}.
Then $|\Lambda|=n$.
\end{lemma}

\begin{proof}
$|\Lambda|\geq n$, because $\{y^*_\alpha: \alpha\in  \Lambda\} $ is a $1$-norming subset. Assume that $|\Lambda|> n$. We are going to demonstrate that every subset of $\{y_\alpha: \alpha\in \Lambda\}$ consisting of $n+1$ elements is linearly independent.

Without loss of generality we can take a subset $\{y_i\}^{n+1}_{i=1}\subset\{y_\alpha: \alpha\in \Lambda\}$. Consider the corresponding linear combination $\sum \limits_{i=1}^{n+1} a_i y_i$ with $\max|a_i|=1$ and let us check that $\sum \limits_{i=1}^{n+1} a_i y_i\neq 0$. Let $j\leq n+1$ be a number such that $|a_j|=1$. Then we can estimate
\begin{align*}
\norm{\sum \limits_{i=1}^{n+1} a_i y_i}\geq \left|y_j^*\left(\sum \limits_{i=1}^{n+1} a_i y_i\right)\right|=\left|a_j y_j^*(y_j)+ \sum \limits_{\substack{i=1 \\ i\neq j}}^{n+1} a_i y_j^*(y_i)\right|\geq 1 -\sum \limits_{\substack{i=1 \\ i\neq j}}^{n+1} |a_i| \rho>0.
\end{align*}
It follows that $Y^{(n)}$ contains $n+1$ linearly independent elements. This contradiction completes the proof.
\end{proof}

\begin{proof}[Proof of Proposition \ref{prop_beta}]
According to  Definition \ref{def:beta} together with Lemma \ref{prop_beta_lemma} there are two sets $\{y_i\}_{i=1}^{n}\subset S_{Y^{(n)}}$, $\{y^*_i\}_{i=1}^{n}\subset S^*_{Y^{(n)}}$ such that \\
$y^*_i(y_i) = 1,$\\
$|y^*_i(y_j)| < 1/n \text{ if } i \neq j,$\\
$\norm{y} = \sup\{|y_i^*(y)|: i=1..n\} \text{, for all } y\in Y.$

Let us define the operator $U: Y^{(n)}\rightarrow \ell_\infty^{(n)}$ by the formula:
$$U(y):=(y_1^*(y), y_2^*(y),...,y_n^*(y)).$$
Obviously, $\|U(y)\|=\|y\|$ for all $y\in Y^{(n)}$, so, $U$ is isometry. Since $\dim Y^{(n)} =\dim \ell_\infty^{(n)}$, the operator $U$ is bijective. This means that $Y^{(n)} $ is isometric to $\ell_\infty^{(n)}$, and since $\beta(\ell_\infty^{(n)})=0$, we have that $\beta(Y^{(n)})=0$.
\end{proof}

So, in order to obtain all possible values of parameter $\rho$ we must consider spaces of higher dimensions.
For every fixed dimension $n$ fix a $\rho\in [\frac{1}{n}, 1)$ and denote $Z=Z^{(n)}_\rho$ the linear space $\R^n$ equipped with the norm 
\begin{align}\label{norm1}
\|x\| =\max\left\{|x_1|, | x_2|,...,| x_n|, \frac{1}{\rho n} \left|\sum \limits_{i=1}^{n} x_i \right|\right\}.
\end{align}

\begin{prop}
Let $Z=Z^{(n)}_\rho$ with $n \ge 2$ and $\rho\in [\frac{1}{n}, 1)$. Then
$$\beta(Z)\leq\rho.$$
\end{prop}

\begin{proof} 
Consider two sets: 

$$\left\{y_j =-\frac{1}{n-1+\rho n}\underset{i\neq j}{\sum \limits_{i=1}^{n}}e_i + e_j, \quad z=\rho\sum \limits_{i=1}^{n}e_i \right\}\subset S_Z$$
  $$\left\{y_j^* = e_j, z^*=\frac{1}{\rho n}\sum \limits_{i=1}^{n}e_i\right\}\subset S_{Z^*}.$$

It follows directly from  \eqref{norm1} that the subset $\left\{\{y_j^*\}^n_{i=1}, z^*\right\}$ is $1$-norming. Also,

$y^*_i(y_i)=1, \quad |y^*_j(y_i)| = |-\frac{1}{n-1+\rho n}|\leq \rho, \quad y^*_j(z) = \rho$,
$z^*(z)=1, \quad z^*(y_i) =\frac{1}{n-1+\rho n}\leq \rho$.
\end{proof}

Remark that in all our estimations of $\Phi^S(X,Y,\eps)$ appears the multiplier $\sqrt{2\eps}$. So, in order to measure the behavior of $\Phi^S(X,Y,\eps)$ in $0$, it is natural to introduce the following quantity
\begin{align*}
\Psi(X, Y):=\underset{\eps\rightarrow 0}{\limsup}\frac{\Phi^S(X,Y,\eps)}{\sqrt{2\eps}}.
\end{align*}
 Also  define
\begin{align*}
\Psi(\rho):=\underset{Y:\beta(Y)=\rho}{\sup}\underset{X}{\sup} \,\underset{\eps\rightarrow 0}{\limsup}\Psi(X, Y)
\end{align*}
which measures the worst possible behavior in $0$ of $\Phi^S(X,Y,\eps)$ when $\beta(Y)\leq\rho$.  From Theorem \ref{th:est_beta} we know that
\begin{align*}
\Psi(\rho)\leq \sqrt{\frac{1+\rho}{1-\rho}} .
\end{align*}
Now we will estimate  $\Psi(\rho)$  from below.
\begin{theorem}
\begin{align*}\Psi(\rho)\geq \min\left\{\sqrt{\frac{2\rho}{1-\rho}},1\right\}
\end{align*}
for all values of $\rho\in (0, 1)$.
\end{theorem}
\begin{proof} 
From Theorem \ref{th:est_beta2} we know that  $\Psi(\rho)\geq 1$. So, we have to check that $\Psi(\rho)\geq \sqrt{\frac{2\rho}{1-\rho}}$. 
In order to estimate  $\Psi(\rho)$  from below for small $\eps$ we consider the couple of spaces $(\ell_1^{(2)}, Z^{(n)}_\rho)$.
Denote $z^*=\frac{1}{\rho n}\sum \limits_{i=1}^{n}e_i$ and $\Gamma=\{x\in S_Z : z^*(x)=1\}$.
Consider the point $x=(1-\delta, \delta)$ and the operator $T$ such that 
$$T(e_1) = \rho\sum \limits_{i=1}^{n} e_i \,\,\text{and}\,\, T(e_2) =  t  \sum\limits_{i=1}^{k} e_i 
+\sum\limits_{i=k+1}^{n}e_i,$$
with  $k=\frac{1}{2}n(1-\rho)+1+\theta \in \N$ being the nearest natural to $\frac{1}{2}n(1-\rho)+1$ (so, $|\theta| \leq 1/2$) and
\begin{align}\label{eq_t}
t=-1+\frac{4+4\theta -2n\rho\frac{\eps_0}{\delta}}{n-n\rho+2+2\theta}\,,
\end{align}
where $\delta>0$ will be defined later and $\eps_0 < \eps$.
Then $z^*(T(x))=1-\eps_0>1-\eps$, so $(x, T)\in \Pi^S_\eps(\ell_1^{(2)}, Z^{(n)}_\rho)$. Now we are searching for the best approximation of $(x, T)$ by a pair $(y, F)\in \Pi(\ell_1^{(2)}, Z^{(n)}_\rho)$. As usual, we have two options:

I. We can approximate the point $x$ by $e_1$ and then we can take $F=T$. In this case we get 
\begin{align}\label{eq_point}
\norm{x-y}=2\delta.
\end{align}

II. We can choose  $F$ which attains its norm in all points of the segment $[e_1, e_2]$, and then we can take $y=x$. In this case $F(e_1)$ and $F(e_2)$ must lie in the same face. Besides, if $F(e_1)\not\in \Gamma$,  we have $\norm{T(e_1)-F(e_1)}=1-\rho>\sqrt{2\eps}\sqrt{\frac{2\rho}{1-\rho}} $  for $\eps$ sufficiently small. To obtain better estimation we must have $F(e_1)\in\Gamma$ and, so, $F(e_2)\in\Gamma$. Then
$$\norm{T-F}\geq \norm{T(e_2)-F(e_2)}\geq \underset{h\in \Gamma}{\inf}\norm{T(e_2)-h}.$$

Let us estimate the distance from $T(e_2)$ to the face $\Gamma$. 

If $h=\sum\limits_{i=1}^{n}h_i \in \Gamma$, then $|h_i|\leq 1$ and $z^*(h)=\frac{1}{\rho n}\sum\limits_{i=1}^{n}h_i=1$. So, $\sum\limits_{i=1}^{k}h_i\geq \rho n -(n-k)$, and 
\begin{align*}
\max h_i \geq \frac{1}{k}(\rho n -(n-k))=-1+\frac{4+4\theta}{n(1-\rho)+2+2\theta}.
\end{align*}
Therefore, 
\begin{align}\label{eq_oper}
\norm{T(e_2)-h}\geq \underset{1\leq i\leq k}{\max}|t-h_i|\geq |t- \max h_i|= 
\frac{2n\rho\frac{\eps_0}{\delta}}{n(1-\rho)+2+2\theta}.
\end{align}

Now let us define $\delta$ as a positive solution of the equation:
$$2\delta=\frac{2n\rho\frac{\eps_0}{\delta}}{n(1-\rho)+2+2\theta}.$$
Then $\delta=\frac{1}{2}\sqrt{2\eps_0}\sqrt{\frac{2\rho}{1-\rho + (2+\theta)/n}}$. Denote $C(\eps,\rho,n, \theta):= \sqrt{2\eps}\sqrt{\frac{2\rho}{1-\rho + (2+\theta)/n}}.$ and $C_0=C(\eps_0,\rho,n, \theta)$. So, with this $\delta$ the estimation \eqref{eq_point} gives us 
$$ \norm{x-y}=2\delta=C_0,$$
and the estimation \eqref{eq_oper} gives us 
$$\norm{T-F}\geq \frac{2n\rho\frac{\eps_0}{\delta}}{n(1-\rho)+2+2\theta}=C_0.$$
In that way, we have shown that $\Phi^S(\ell_1^{(2)}, Z^{(n)}_\rho,\eps)\geq C_0$. As $\eps_0$ can be chosen arbitrarily close to $\eps$ we obtain that $\Phi^S(\ell_1^{(2)}, Z^{(n)}_\rho,\eps)\geq C(\eps,\rho,n, \tilde{\theta})$ with $\tilde{\theta}\in [-1/2, 1/2]$. Consequently, we have that $\Psi(\ell_1^{(2)}, Z^{(n)}_\rho)\geq \sqrt{\frac{2\rho}{1-\rho + (2+\tilde{\theta})/n}}$. When $n\rightarrow \infty$, we obtain the desired estimation $\Psi(\rho)\geq\sqrt{\frac{2\rho}{1-\rho}}$.
\end{proof}

\section{Modified Bishop-Phelps-Bollob\'{a}s moduli for operators} \label{sec4}
The following modification of the Bishop-Phelps-Bollob\'{a}s theorem can be easily deduced from Proposition \ref{Corollary2.2Phelps} just by substituting $\eta = \eps, k = \sqrt{\eps}$.

\begin{theorem}[Modified Bishop-Phelps-Bollob\'{a}s theorem]\label{Modifth:BPB}
Let $X$ be a Banach space. Suppose $x\in B_X$ and $x^*\in B_{X^*}$
satisfy $ x^*(x) \geq 1 -  \eps$ ($\eps \in (0, 2)$). Then there exists $(y,y^*)\in S_X\times X^*$ with $\|y^*\|=y^*(y)$ such that
\begin{equation*} 
 \max\{\|x-y\|, \|x^*-y^*\|\} \leq\sqrt{\eps}.
 \end{equation*}
\end{theorem}
The improvement in this estimate comparing to the original version appears because we do not demand
$\|y^*\|=1$.  It was shown in \cite{KadS} that this theorem is sharp in a number of two-dimensional spaces, which makes a big difference with the original  Bishop-Phelps-Bollob\'{a}s theorem, where the only (up to isometry) two-dimensional space, in which the theorem is sharp, is $\ell_1^{(2)}$. Bearing in mind this theorem it is natural to introduce the following quantities.
\begin{definition}
 \textit{The  modified Phelps-Bollob\'as modulus} of a pair $(X, Y)$ is the function, which is determined by the following formula:
\begin{align*}
\widetilde{\Phi}(X,Y,\varepsilon)= \inf \lbrace \delta>0: \text{for every} \, T\in B_{L(X,Y)} ,  \mbox{ if } x\in B_X \mbox{ and } \|T(x)\|> 1-\eps, \, \mbox{ then there exist} \\
y\in S_X \mbox{ and } F\in L(X,Y)  \mbox{ satisfying }       
\, { \|F(y)\|=\norm{F}}, \|x-y\|<\delta \text{ and } \|T-F\|< \delta \rbrace.
\end{align*}

\textit{The modified spherical Bishop-Phelps-Bollob\'as modulus} of a pair $(X, Y)$ is the function, which is determined by the following formula:
\begin{align*}
\widetilde{\Phi^S}(X,Y,\varepsilon)= \inf \lbrace \delta>0: \text{for every} \, T\in S_{L(X,Y)} ,  \mbox{ if } x\in S_X \mbox{ and } \|T(x)\|> 1-\eps, \, \mbox{ then there exist} \\
y\in S_X \mbox{ and } F\in L(X,Y)  \mbox{ satisfying }       
\, \|T(y)\|=\norm{T}, \|x-y\|<\delta \text{ and } \|T-F\|< \delta\rbrace.
\end{align*}

\end{definition}

By analogy with Theorem \ref{th:est_beta} we prove the next result. 

\begin{theorem}\label{modif_th:est_beta}
Let $X$ and $Y$ be Banach spaces such that $Y$ has the property $\beta$ with parameter $\rho$. Then the pair $(X,Y)$ has the Bishop-Phelps-Bollob\'as property for operators and for any $\varepsilon\in(0, 1)$
\begin{equation} \label{equat:modif_th:est_beta}
\widetilde{\Phi^S}(X,Y,\varepsilon)\leq\widetilde{\Phi}(X,Y,\varepsilon)\leq \min \left\lbrace\sqrt{\eps}\,\sqrt{\frac{1+\rho}{1-\rho}}, \,1\right\rbrace .
\end{equation}
\end{theorem}
The proof is similar to Theorem \ref{th:est_beta} but it has some modifications and we give it here for the sake of clearness.

\begin{proof}
Consider $T\in B_{L(X,Y)}$ and $x\in B_X$ such that $\|T(x)\|> 1-\eps$ with $\eps\in\left(0, \frac{1-\rho}{1+\rho}\right)$.
Since $Y$ has the property $\beta$, there is $\alpha_0\in \Lambda$ such that $|y^*_{\alpha_0}(T(x))|> 1-\eps$.
So, if we denote $x^*=T^*y^*_{\alpha_0}$, we have $x\in B_X, x^*\in B_{X^*}$ with $x^*(x)>1-\eps$. We can apply the  formula \eqref{eqBPB1} from Lemma \ref{LemmaBPB}, for
any  $k \in (0, 1)$. For every $\tilde{k}\in [\eps, 1)$ let us take 
$$k=\frac{ \tilde{k} (\norm{x^*} - (1-\eps))}{\eps \norm{x^*}}.$$
The inequality $\norm{x^*} \ge x^*(x) > 1-\eps$ implies that $k > 0$. On the other hand, $k=\tilde{k}\left(\frac{1}{\eps} - \frac{(1-\eps)}{\eps  \norm{x^*}}\right)\le \tilde{k}\left(\frac{1}{\eps} - \frac{(1-\eps)}{\eps}\right) = \tilde{k} < 1$, so for this $k$ we can find  $\zeta^*\in X^*$ and $z\in S_X$  such that 
there exist $z^*\in X^*$ and $z\in S_X$ such that $|z^*(z)|=\norm{z^*}$ and 
\begin{align*}
\|x-z\|< \frac{1-\frac{1-\eps}{\norm{x^*}}}{k} \,\,\text{and}\,\,\|z^*-x^*\|<k\norm{x^*}.
\end{align*}

For a real number $\eta$ satisfying $\eta> \frac{\rho(k\|x^*\|+\|x^*\|\cdot|1-\|z^*\||)}{\|z^*\|(1-\rho)}$ we define the operator $S\in L(X,Y)$ by the formula
$$S(t) = \norm{z^*} T(t)+[(1+\eta)z^*(t)-\norm{z^*} T^*(y^*_{\alpha_0})(t)]y_{\alpha_0}.$$
 Let us estimate the norm of $S$. Recall that we denote $x^*=T^*y^*_{\alpha_0}$. Thus for all $y^*\in Y^*$,
$$
S^*(y^*) = \norm{z^*}T^*(y^*)+[(1+\eta)z^* -\norm{z^*} x^*]y^*(y_{\alpha_0}).
$$
Since the set $\{y^*_{\alpha}: \alpha \in \Lambda\}$ is norming for $Y$ it follows that $\norm{S} = \sup_\alpha \norm{S^*{y^*_{\alpha}}}$.
$$\norm{S} \geq \norm{S^*(y^*_{\alpha_0})} = (1+\eta)\norm{z^*}.$$
On the other hand for $\alpha \neq \alpha_0$ we obtain
$$\norm{S^*(y^*_{\alpha})}\leq \norm{z^*} +\rho[\|z^*-x^*\|+\|x^*\|\cdot |1-\|z^*\|| + \eta\|z^*\|]\leq(1+\eta)\norm{z^*}.$$
Therefore, 
$$\norm{S} = \norm{S^*(y^*_{\alpha_0})} = \norm{z^*}=|y^*_{\alpha_0}(S(z))|\leq \norm{S(z)}\leq\norm{S}.
$$
So, we have $\|S\|=\|S(z)\|=(1+\eta)\norm{z^*}$.

Let us estimate $\|S-T\|$.
$$\norm{S-T} = \sup_\alpha \norm{S^*y^*_{\alpha}-T^*y^*_{\alpha}}.$$ 
Notice also that
\begin{align*}|1-\norm{z^*}|\leq \norm{x^*-z^*} + 1 - \norm{x^*}< k\|x^*\|+ 1 - \norm{x^*}.
\end{align*}
For $\alpha = \alpha_0$ we get 
\begin{align*}
\norm{S^*y^*_{\alpha_0}-T^*{y^*_{\alpha_0}}}=\|(1+\eta)z^*-x^*\|
\leq\|z^*-x^*\|+\eta\|z^*\|\\
<\frac{k\|x^*\|(1+\rho\|x^*\|)+\rho\|x^*\|(1-\|x^*\|)}{1-\rho}.
\end{align*}
Then $\|x-z\|<\frac{\eps}{\tilde{k}}$ and
\begin{align*}
\norm{S^*y^*_{\alpha_0}-T^*{y^*_{\alpha_0}}}
<\frac{\tilde{k}\frac{\|x^*\|-(1-\eps)}{\eps}(1+\rho\|x^*\|)+\rho\|x^*\|(1-\|x^*\|)}{1-\rho}
\leq\frac{\tilde{k}(1+\rho)}{1-\rho}.
\end{align*}
The last inequality holds, because  if we consider the function $$f(t)=\frac{\tilde{k}\frac{t-(1-\eps)}{\eps}(1+\rho t)+\rho t(1-t)}{1-\rho}$$ with $t\in (1-\eps, 1)$, then ${f'}\geq 0$, so $\max f = f(1)=\frac{\tilde{k}(1+\rho)}{1-\rho}$.  For $\alpha \neq \alpha_0$ we obtain
\begin{align*}
\norm{S^*y^*_{\alpha}-T^*{y^*_{\alpha}}}&\leq |1-\|z^*\||+\rho (\|z^*-x^*\| +\|x^*\|\cdot|1-\|z^*\||+\eta \|z^*\| )\\
&< k\|x^*\| + 1-\|x^*\|+\frac{\rho}{1-\rho}[k\|x^*\|-\rho k \|x^*\|+\|x^*\|\cdot|1-\|z^*\|| \\
&-\rho\|x^*\|\cdot|1-\|z^*\||+\rho k \|x^*\|+ \rho\|x^*\|\cdot|1-\|z^*\||]\\
&\leq k\|x^*\| + 1-\|x^*\|+\frac{\rho}{1-\rho}[k\|x^*\| +\|x^*\|\cdot(k\|x^*\|+1-\|x^*\|)].
\end{align*} 
Substituting the value of ${k}$ we get 
\begin{align*}
\norm{S^*y^*_{\alpha}-T^*{y^*_{\alpha}}}&< \frac{1-\rho(1-\|x^*\|)}{1-\rho}\left[\tilde{k}\frac{\|x^*\|-(1-\eps)}{\eps}+1-\|x^*\|\right]\\
&+\frac{\rho \tilde{k}}{1-\rho}\frac{\|x^*\|-(1-\eps)}{\eps}\leq\frac{\tilde{k}(1+\rho)}{1-\rho}.
\end{align*}

To get the last inequality we again use the fact that the function 
$$f_1(t)=\frac{1-\rho(1-t)}{1-\rho}\left[\tilde{k}\frac{t-(1-\eps)}{\eps}+1-t\right]
+\frac{\rho \tilde{k}}{1-\rho}\frac{t-(1-\eps)}{\eps}$$ 
 is increasing if $\tilde{k} \geq\eps$, so $\max f_1 = f_1(1)=\frac{\tilde{k}(1+\rho)}{1-\rho}$. So, $\|T-S\|\leq\frac{\tilde{k}(1+\rho)}{1-\rho}$.
 
Let us substitute $\tilde{k}=\sqrt{\frac{\eps(1-\rho)}{1+\rho}}$ (here we need $\eps<\frac{1+\rho}{1-\rho}$ which holds for any $\eps\in (0,1)$ and also $\eps<\frac{1-\rho}{1+\rho}$). Then we obtain 
$$\max\{\|z-x\|, \|T-S\|\}\leq\sqrt{\frac{\eps(1+\rho)}{1-\rho}}.$$
Finally, if $\eps\geq\frac{1-\rho}{1+\rho}$, we can always approximate $(x, T)$ by the same point and an zero operator, so $\max\{\|z-x\|, \|T-S\|\}\leq 1.$
\end{proof}

The above theorem implies that if $\beta(Y)=0$, then $\widetilde{\Phi^S}(X,Y,\varepsilon)\leq\widetilde{\Phi}(X,Y,\varepsilon)\leq\sqrt{\varepsilon}$. We are going to demonstrate that this estimation is sharp for $X=\ell^{(2)}_1, Y=\R$.

\begin{theorem}
\quad $\widetilde{\Phi^S}(\ell^{(2)}_1,\R,\varepsilon) = \widetilde{\Phi}(\ell^{(2)}_1,\R,\varepsilon) =\sqrt{\varepsilon} , \,\eps\in(0, 1).$
\end{theorem}
\begin{proof}
We must show that for every $0<\eps<1$ and for every $\delta> 0$ there is a pair $(x, x^*)$ from $\Pi^S_\eps(\ell^{(2)}_1, \R)$ such that for every pair $(y, y^*)\in S_{\ell^{(2)}_1}\times \ell^{(2)}_\infty$ with $|y^*(y)|=\norm{y^*}$
\begin{equation}\label{mod_eq2}
\max\left\{ \norm{x-y}, \norm{x^*-y^*}\right\}\geq \sqrt{\varepsilon}-\delta.
\end{equation}

Fix an $\eps_0<\eps$ such that $\sqrt{\eps_0}>\sqrt{\eps}-\delta$. Take a point $x:=\left( 1-\frac{\sqrt{\eps_0}}{2}\right)e_1 +\left( \frac{\sqrt{\eps_0}}{2} \right)e_2,$
and a functional 
$x^*(z):=z_1+(1-2\sqrt{\eps_0})z_2.$
Notice that $x^*(x)=1-\eps_0>1-\eps$. 

Consider the set $U$ of those $y \in S_X$ that  $\|x-y\| <\sqrt{\eps_0}$. $U$ is the intersection of $S_X$ with the open ball of radius $\sqrt{\eps_0}$  centered in $x$. 
As $\norm{x-e_1}=\sqrt{\eps_0},$ and $\norm{x-e_2}=2 - \sqrt{\eps_0}\geq\sqrt{\eps_0},$ so, $U\subset]e_1, e_2[$, and for every $y=ae_1+be_2 \in U$ $a>0$ and $b>0$.

 Assume that $|y^*(y)|=\|y^*\|$ for some $y\in U$ and $\norm{x^*-y^*}\leq\sqrt{\eps_0}$. Then we are forced to have $y^*=(y^*(e_1), y^*(e_2))$, where $|y^*(e_1)|=|y^*(e_2)| $ and $y^*(e_1)\cdot y^*(e_2)\geq 0$. Notice that  
$$|x^*(e_1)-y^*(e_1)|=|1-y^*(e_1)|\leq \norm{x^*-y^*}\leq \sqrt{\eps}\Rightarrow y^*(e_1)\geq 1-\sqrt{\eps_0},$$
$$|x^*(e_2)-y^*(e_2)|=|1-2\sqrt{\eps_0}-y^*(e_2)|\leq \norm{x^*-y^*}\leq \sqrt{\eps_0}\Rightarrow y^*(e_2)\leq 1-\sqrt{\eps_0}.$$
Then $y^*=(1-\sqrt{\eps_0}, 1-\sqrt{\eps_0})$, so $\norm{x^*-y^*}=\sqrt{\eps_0}>\sqrt{\eps}-\delta$. It follows that inequality \eqref{mod_eq2} holds, as desired. 
\end{proof}

Also with the same space $Y = Y_\rho$  equipped with the norm  \eqref{eq_norm} we have an estimation from below which almost coincides with the estimation \eqref{equat:modif_th:est_beta} from above for values of $\rho$ close to 1.
\begin{theorem} Let $ \rho \in [1/2, 1)$, $0<\eps< 1$. Then, in the space $Y = Y_\rho$ 
$$
\widetilde{\Phi^S}(\ell_1^{(2)},Y,\varepsilon) \geq \min \left\lbrace\sqrt{\eps}\,\sqrt{\frac{2\rho}{1-\rho}}, \,1\right\rbrace.
$$
\end{theorem}

\section{An open problem} \label{sec5}

\begin{probl}
Is it true, that $\Phi^S(X,\R,\varepsilon)\leq\min\{\sqrt{2\varepsilon}, 1\}$ for all real Banach spaces $X$?
\end{probl}
 
 In order to explain what do we mean, recall that for the original Bishop-Phelps-Bollob\'{a}s modulus the estimation 
\begin{align}\label{eq_functionals}
\Phi_X^S(\varepsilon)\leq\sqrt{2\varepsilon}
\end{align}
holds true for all $X$. In other words, for every $(x,x^*)\in S_X\times S_{X^*}$ with $x^*(x)>1-\eps$, there is $(y,y^*)\in S_X\times S_{X^*}$ with $y^*(y)=1$ such that
$\max\{\|x-y\|<\sqrt{2\varepsilon}, \|x^*-y^*\|\}<\sqrt{2\varepsilon}$.

When we take $Y=\R$ in the definition of $\Phi^S(X, Y, \varepsilon)$, the only difference with 
$\Phi_X^S(\varepsilon)$ is that by attaining norm we mean $|y^*(y)|=1$,
instead of $y^*(y)=1$.  So, in the case of $\Phi^S(X,\R,\varepsilon)$ we have more possibilities to approximate  $(x,x^*)\in S_X\times S_{X^*}$ with $x^*(x) >1-\eps$: 
$$(y,y^*)\in S_X\times S_{X^*} \,\text{with}\, y^*(y)=1 \,\text{or}\, y^*(y)=-1.$$
Estimation \eqref{eq_functionals} is sharp for the two-dimensional real $\ell_1$ space:
\begin{align}\label{eq_problem1}
\Phi_{\ell^{(2)}_1}^S(\varepsilon)=\sqrt{2\varepsilon},
\end{align}
but, as we have shown in Theorem \ref{th:est_beta2}  
\begin{align}\label{eq_problem2}
\Phi^S(\ell^{(2)}_1,\R,\varepsilon)=\min\{\sqrt{2\varepsilon}, 1\}.
\end{align}

Estimations \eqref{eq_problem1} and \eqref{eq_problem2} coincide for $\eps\in (0, 1/2)$, but for bigger values of $\eps$ there is a significant difference. We do not know whether the inequality $\Phi^S(X,\R,\varepsilon)\leq\min\{\sqrt{2\varepsilon}, 1\}$ holds true for all $X$. 

Moreover, in all examples that we considered we always were able to estimate $\Phi^S(X, Y, \eps)$ from above by $1$. So, we don't know whether the statement of Theorem \ref{th:est_beta} can be improved to  
$$\Phi^S(X,Y,\varepsilon)\leq\min\left\{ \sqrt{2\varepsilon}\sqrt{\frac{1+\rho}{1-\rho}},\,\, 1\right\}.$$

\end{document}